\documentclass[10pt,oneside]{amsart}
\usepackage{amssymb, amscd, amsmath, amsthm, latexsym, hyperref,ams refs}
\usepackage{pb-diagram, fancyhdr, graphicx, psfrag, enumerate}
\usepackage[text={6.3in,8.5in},centering,letterpaper,dvips]{geometry}
\usepackage{enumitem}\setlist[enumerate,1]{font=\upshape, itemsep=1ex}\setlist[itemize,1]{font=\upshape, itemsep=1ex}
\usepackage{pinlabel}
\usepackage{xcolor}

\def\C{{\mathbb C}}

\def\calc{\mathcal{C}}

\def\cs{\mathbin{\#}}

\newcommand{\spinc}{\ifmmode{{\mathfrak s}}\else{${\mathfrak s}$\ }\fi}
\newcommand{\spinct}{\ifmmode{{\mathfrak t}}\else{${\mathfrak t}$\ }\fi}
\newcommand{\spincr}{\ifmmode{{\mathfrak r}}\else{${\mathfrak r}$\ }\fi}

\newcommand{\fig}[2] { \includegraphics[scale=#1]{#2} }

\newtheorem{theorem}{Theorem} 
\newtheorem{lemma}[theorem]{Lemma}
\newtheorem{corollary}[theorem]{Corollary}

\theoremstyle{definition}
\newtheorem{problem}{Problem}

\newtheorem{example}[theorem]{Example}

\begin{document}
\title{The cobordism distance between a knot and its reverse}
 
\author{Charles Livingston}
\thanks{This work was supported by a grant from the National Science Foundation, NSF-DMS-1505586.   }
\address{Charles Livingston: Department of Mathematics, Indiana University, Bloomington, IN 47405}\email{livingst@indiana.edu}
 


\begin{abstract}  We consider the question of how knots and their reverses are related in the concordance group $\calc$.  There are examples of knots for which $K \ne K^r \in \calc$.  This paper studies the  cobordism distance, $d(K, K^r)$.  If $K \ne K^r \in \calc$, then   $d(K, K^r) >0$ and it  elementary to see that for all $K$, $d(K, K^r) \le 2g_4(K)$.  It is known that $d(K,K^r)$ can be arbitrarily large.
Here we present a  proof  that for non-slice  knots satisfying $g_3(K) = g_4(K)$, one has $d(K,K^r) \le 2g_4(K) -1$.  This family includes all strongly quasi-positive knots and all non-slice genus one knots.  We also construct knots $K$ of arbitrary four-genus  for which   $d(K,K^r) = g_4(K)$.  Finding knots for which $d(K,K^r) > g_4(K)$ remains an open problem.

 \end{abstract}

\maketitle


\section{Introduction}
  
 This work explores concordance relationships between knots and their reverses.  Recall that if a knot $K$ is defined  formally as an oriented pair $(S^3, K)$, then 
 $K^r $ is shorthand for   $(S^3, -K)$.   The results here apply equally in the smooth and in the topological locally flat category.
 
 The problem of distinguishing a knot $K$ from $K^r$ did not receive much attention until Fox noted the difficulty of the problem in~\cite[Problem 10]{MR0140100}.  However,   from the perspective of  the classification of knots, the distinction is essential: the connected sum of knots  is not well-defined outside of the oriented category.  
 Advances in the study of knot reversal include~\cites{MR683753, MR0158395, MR1395778}. 
 
  Distinguishing the concordance classes of $K$ and $K^r$ is more difficult; see~\cite{MR1670424} for an example in which Casson-Gordon invariants, realized as twisted Alexander polynomials, are applied to the problem.  Infinite families of topologically slice knots for which $K \ne K^r$ in smooth concordance are produced in~\cite{2019arXiv190412014K}.  The issue appears  in current research concerning satellite operations.  Given a knot $P \subset S^1 \times B^2$, there is a self-map of  the concordance group $\calc$  given by using $P$ to form  a satellite, often denoted  $K \to  P(K)$, where here $K$ and $P(K)$ denote smooth concordance classes.  In general, $P(K) \ne P(K^r)$,  $P(K) \ne P^r(K)$, and $P^r(K^r) \ne P(K)$.    Work about   such satellite constructions 
  includes~\cite{MR3848393,2018arXiv180904186H,MR3286894,MR3784230}.
  
  Being independent in $\calc$ is only one measure of knots being distinct.  Here we consider the cobordism distance, which is a metric on $\calc$:  $d(K,J)$ is the minimum genus of a cobordism from $K$ to $J$; it equals the four-ball genus, $g_4(K \cs -J)$.  Research about the cobordism distance  includes a careful analysis of torus knots.  For instance, for distinct positive parameters, Litherland~\cite{MR547456} proved  that  the torus knots $T(p,q)$ and $T(p',q')$ are linearly independent.  On the other hand, Feller-Park~\cite{2019arXiv191001672F}    analyzed the question of determining for which pairs $ d(T(p,q), T(p',q'))= 1$, resolving all but one case: $d(T(3,14), T(5,8))$.  Related  work includes~\cite{MR2975163,MR3622312}
  
 The triangle inequality states that for all $K$ and $J$, $d(K,J) \le g_4(K) + g_4(J)$.     It seems possible that   in some generic sense, one usually has an equality.  In particular, one might suspect that in general  $d(K,K^r ) = 2g_4(K)$.  Of course, this wouldn't appear for low crossing number knots,  most of which are reversible.  Here we present a simple construction to prove  that for a large class of knots  this equality does not hold.
 
 \begin{theorem}\label{thm:main}
 If $K$ is nontrivial,  then $d(K, K^r) < 2g_3(K)$.
 \end{theorem} 
 
  \begin{corollary}\label{corollary:main}
 If $g_3(K) = g_4(K) \ne 0$,  then $d(K, K^r) < 2g_4(K)$.
 \end{corollary} 
 
 Here are a few  observations about Corollary~\ref{corollary:main}.
 \begin{enumerate}
 \item The hypothesis hold for all strongly quasi-positive  knots~\cite{MR1193540}.
 
 \item The hypothesis holds for all non-slice knots for which $g_3(K) = 1$.
 
 \item The theorem and corollary can be strengthened by replacing the condition  $g_3(K) = g_4(K)$ with the  condition on the concordance genus, $g_c(K) = g_4(K)$. 
 
 \item  There are no known  examples of a non-slice knot for which  $d(K, K^r) = 2g_4(K)$.  In fact, there are no examples known for which  $d(K, K^r) >g_4(K)$.
 
 \end{enumerate}
 
 The proof of Theorem~\ref{thm:main} is in Section~\ref{section:proof1}.  It generalizes an observation made in~\cite{MR2195064} about  a specific family of genus one knots.  In Section~\ref{section:exproof1} we discuss extensions of this result.

 Section~\ref{section:examples} presents examples of knots   with arbitrarily large four-genus 
 for which  $d(K, K^r) = g_4(K)$, here stated formally:
 
 \begin{theorem}\label{thm:g4}  For every integer $g\ge0$, there exists a knot $K$ for which $d(K, K^r) = g_4(K) = g$.
 \end{theorem}  The proof uses techniques from the work of Gilmer~\cite{MR656619} that applied Casson-Gordon theory to study the four-genus  of knots.  It slightly refines a result in~\cite{MR2195064} where it was shown that $g_4(K \cs -K^r)$  can be arbitrarily large.

 \medskip
 \noindent{\it Acknowledgments.}  Thanks to JungHwan Park for providing valuable feedback on an early version of this paper.  
 
 \section{Proof of Theorem~\ref{thm:main}}\label{section:proof1} 
 
 A basic knot theory reference, such as the text by Rolfsen~\cite{MR0515288}, is sufficient for the results used in the proof here.  A schematic illustration for the  proof is presented   in Figure~\ref{fig:schematic}.

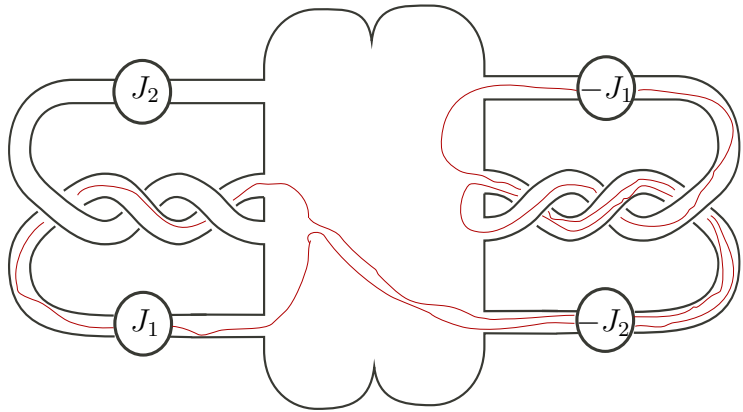
\begin{figure}[h]
\labellist
\pinlabel {\text{\Large{$J_1$}}} at 47 30
\pinlabel {\text{\Large{$J_2$}}} at 47 110
\pinlabel {\text{\Large{$-J_2$}}} at 204 30
\pinlabel {\text{\Large{$-J_1$}}} at 205 110
\endlabellist
\includegraphics[scale=1.1]{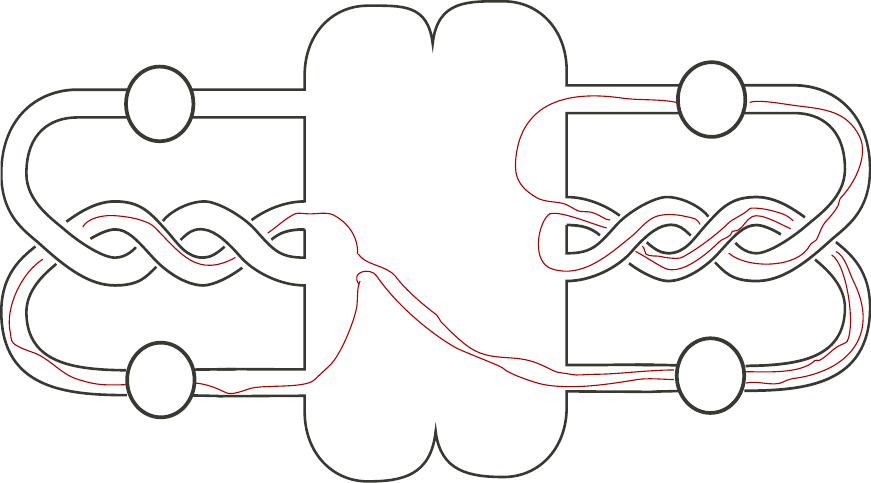} 

\caption{A schematic diagram of the proof of Theorem~\ref{thm:main}, illustrating the band sum that creates the surgery curve on $F \cs F^*$ having boundary $K \cs -K^r$.  In this example, $g_3(K) = 1$.  The thin (red) curve represents $\alpha \cs_b \alpha^*$ in the notation of the proof.}
\label{fig:schematic}
\end{figure}

Let $F$ be a minimal genus Seifert surface for $K$.  Denote its  genus by $g$.  In Figure~\ref{fig:schematic} the genus of $F$ is 1.   Let  $\alpha$ be an oriented simple closed curve on $F$ representing a nontrivial homology class.   We will also think of $\alpha$ as representing a knot. (In the figure, $\alpha$ has the knot type of $J_1$.)
 
 There is a natural choice of Seifert surface for $K \cs -K^r$, which we will denote $F \cs F^*$.     We denote by $\alpha^*$ the curve of $F^*$ that corresponds to $\alpha$.  One must specify  an orientation on $\alpha^*$, but that  choice will not be relevant in the proof.  What is essential is the elementary observation that with some orientation, $\alpha^* = - \alpha$ as a knot.
 
   If $V$ is the Seifert form for $ F \cs F^*$, then it is straightforward to show  $V(\alpha, \alpha) = -V(\alpha^*, \alpha^*)$.   Notice that reversing a knot on a Seifert surface  does not change its  self-linking.

 Consider the abstract  surface $\widehat{G}$ formed by cutting $F \cs F^*$ along $\alpha$ and $\alpha^*$.  It has five boundary components:  one that corresponds to 
 $K \cs -K^r$; a pair   $\{\alpha_1$, $\alpha_2\}$ corresponding to the $\alpha$ cut, and a pair $\{\alpha^*_1$,  $\alpha^*_2\}$  that corresponds to the $\alpha^*$ cut.  Choose simple paths $\gamma_{i}$ from $\alpha_1$ to  $\alpha^*_i$.  Each  $\gamma_{i}$, viewed as a curve in $S^3$, can be  used to form the band connected sum of $\alpha$  and $\alpha^*$ with the band in $F \cs F^*$.  Orientations can be chosen so that one is  $\alpha \cs_b  \alpha^*$ and the other is  $\alpha \cs_b  \alpha^{*r}.$    Thus, we can choose one of the bands to find $\alpha \cs_b -\alpha$ embedded on $F \cs F^*$ as a homologically nontrivial curve.
 
 The curve $\alpha \cs_b -\alpha$ has Seifert framing 0. If it is slice, then the Seifert surface $F \cs F^*$ can be surgered in the four-ball to yield a surface of genus  $2g-1$ bounded by $K \cs -K^r$ in $B^4$.  The proof can now be completed in one of two ways.   One option is to observe that for a natural choice of band,   $\alpha \cs_b -\alpha = \alpha \cs -\alpha$.  As an alternative, one can  use  the following result of Miyazaki, for which we include a short summary proof.  
 \begin{lemma}\label{lemma:1} If $\{K,  J\}$ is a split two component link, then the concordance class of the band connected sum  $K \cs_b  J$ is independent of the choice of band.
 \end{lemma}
 
 \begin{proof} The connected sum of knots is a special case of the band connected sum, so the result will be proved by showing that  $(K \cs_b  J) \cs -(K \cs J)$ is slice.  In this sum, one can slide  the knot $-K$  over the band  to see that  
 \[(K \cs_b  J) \cs -(K \cs J)   \cong 
  (K \cs -K) \cs_b (J \cs -J).
  \]
 It is now evident that this knot is slice: a band move (dual to the band $b$) splits it into two components forming an split link consisting of  $K \cs -K$ and $J \cs -J$, both of which are slice.
 \end{proof}


 \section{Generalizations}\label{section:exproof1}
 
 The proof of Theorem~\ref{thm:main} could have been based on the following lemma, which can also be applied in some cases to reduce the genus of the bounding surface to be less than $2g_3(K) -1$.  The proof of Lemma~\ref{lemma:2} is much the same as that of Lemma~\ref{lemma:1} so we do not include it here.  Similarly, the proof of Corollary~\ref{corollary:1} follows along the lines of that of Theorem~\ref{thm:main} and Corollary~\ref{corollary:main}.
 
  \begin{lemma}\label{lemma:2}  Suppose that $\{J_1, \ldots ,J _n\}$ and $\{J'_1, \ldots , J'_n\}$ are split links.  Then for any set of bands $\{b_i\}$ joining each $J_i$ to $J_i'$, the band  connected sum is  concordant to the split link with compoments $J_i \cs J_i'$.
 
  \end{lemma}
  
  \begin{corollary}\label{corollary:1} Suppose that $g_3(K) = g_4(K) \ne 0$ and that $F$ is a minimum genus Seifert surface for $K$.  If a rank $k$ subgroup of $H_1(F)$ is represented by a set of $k$ disjoint simple closed curves on $F$ that form a split link, then $d(K, K^r) \le 2g_4(K) - k$.
  \end{corollary}
  
  As an example  we use pretzel knots, which include the first knots shown to not be reversible~\cite{MR0158395}.

  \begin{example} Consider the $n = 2k +1 $ stranded pretzel knot, with $n$ odd: $P = P(p_1, \ldots , p_n)$.  Assume that $p_i >0$ and each $p_i$ is odd.  Then $g_3(P) = g_4(P) = k$ and $d(P, P^r) \le k$.  To prove  this, one considers the Seifert matrix $V_P$ of dimension $2k \times 2k$.  Since the $p_i$ are positive, the matrix is easily seen to be diagonally dominant, and thus the signature satisfies $\sigma(P) = 2k$.  It follows that $g_4(P) \ge k$, but clearly $g_3(P) \le k$.  Using the fact that $g_4(P) \le g_3(P)$ leads to the observation that $g_3(P) = g_4(P) = k$.   There is an obvious set of $k$ curves on the standard Seifert surface for $P$ that represent independent classes in homology, and these form an unlink. 
  \end{example}

  \begin{example} More complicated examples can be built from the knot $P(p_1, \ldots , p_n)$   by tying arbitrary knots in the bands of its canonical Seifert surface.
  \end{example}


 \section{Examples of $g_4(K \cs -K^r) = g_4(K)$.}\label{section:examples}
 
 In this section we prove Theorem~\ref{thm:g4}.  Our approach modifies examples built in~\cite{MR2195064}, and again our proof depends on results
  of Gilmer~\cite{MR656619}  that demonstrated that Casson-Gordon invariants offer bounds on the four-genus of knots.  The result in~\cite{MR2195064} is sufficient to find knots for which $g_4(K \cs -K^r) \ge  g_4(K)$; the added feature here is that for the current examples the four-genus $g_4(K)$ is exactly identified.
 
Figure~\ref{fig:knot} is a schematic representation of a knot we denote $K(A,B)$.  In the figure there is a  link $(K, \alpha, \beta)$.  If we replace  neighborhoods of  $\alpha$ and $\beta$ with the complements of knots $A$ and $B$ (identifying the meridian and longitudes of $A$ and $B$ with the longitudes and meridians of $\alpha$ and $\beta$, respectively), the resulting  manifold is diffeomorphic to $S^3$.   The knot $K$ now represents the knot we call $K(A,B)$.  In effect, the knots $A$ and $B$ are being tied in the two bands of the obvious Seifert surface.

 \begin{figure}[h]
\labellist
\pinlabel $K$ at -30 100
\pinlabel $\alpha$ at 10 365
\pinlabel $\beta$ at 640 365
\endlabellist
\fig{.2 }{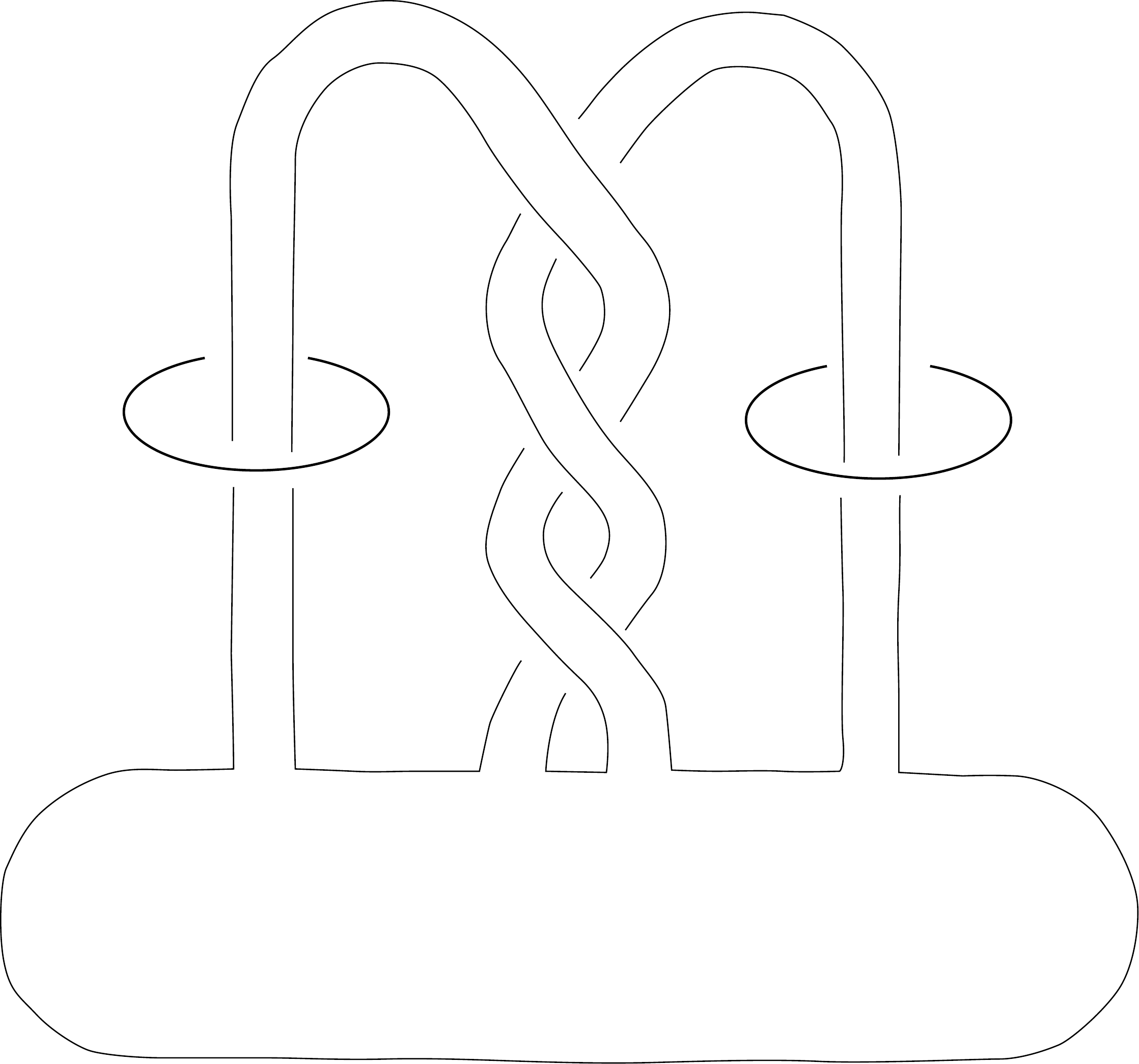}  
\caption{Knot $K(A,B)$}
\label{fig:knot}
\end{figure}

We will prove the following.

\begin{theorem}  For every $n \ge 0$ there is a pair of knots $A$ and $B$ such that \[g_4(nK(A,B)) = n = g_4( nK(A,B) \cs -nK(A,B)^r) .\]
\end{theorem}

The difficult work is in showing that $n$ is a lower bound for both invariants.  We begin with the simple fact that it is an upper bound.

\begin{lemma} For any choice of $A$ and $B$,  $g_4( n K(A,B) )\le  n$ and  $g_4\big( n  K(A,B) \cs -nK(A,B)^r \big)\le  n$.
\end{lemma}

\begin{proof} We have that $g_3\big(K(A,B)\big) =1$.  Thus $g_4\big(  K(A,B)\big) \le 1$ and it follows that $g_4\big(nK(A,B)\big) \le  n$.  By Theorem~\ref{thm:main}, $g_4\big( K(A,B) \cs -K(A,B)^r\big) \le 1$, and thus   $g_4\big( nK(A,B) \cs -nK(A,B)^r)\big) \le n$, 
\end{proof}

The remainder of this section is devoted to reversing these inequalities.

 \subsection{Signature bounds related to $g_4\big( n K(A,B)\big)$}

In~\cite{MR656619} Gilmer considered the case of $A = J = B$ for some $J$.    However, his result~\cite[Theorem 1]{MR656619} can be applied directly to achieve the next result.

\begin{theorem}\label{thm:12} If $g_4\big(nK(A,B) \big) = g < n$, then for some $a$ and $b$ satisfying $0 \le a \le n$, $0 \le b \le n$ and at least one of $a$ and $b$ are nonzero, one has:
 \[ \big| \sum_{i=1}^{n} \sigma_\lambda\tau(K, \chi_i) +
  2a\sigma_{  1/3}(A) +2b\sigma_{  1/3}(B) \big| \le 4g. \] 
\end{theorem}
We will not define the quantity  $\sigma_\lambda\tau(K, \chi_i)$; all that is needed is that it is independent of the choice of $A$ and $B$, that there are only a finite number of possible $\chi_i$ that can arise, and that for each $\chi$, the values of $\sigma_\lambda\tau(K, \chi_i)$ are uniformly bounded by some constant as functions of $\lambda \in S^1 \subset \C$.

 \subsection{Signature bounds related to $g_4\big( n  K(A,B) \cs -nK(A,B)^r\big)$}

 The proof of the main theorem in~\cite[Section 5]{MR2195064} yields the following result, stated in terms of the classical Levine-Tristram signatures.

\begin{theorem}\label{thm:general}   If  $g_4\big(n(K(A,B) \cs -K(A,B)^r)\big) = g < n$, then there exists a pair of integers $a$ and $b$ satisfying $0 \le a \le n$, $0 \le b \le n$, and at least one of $a$ and $b$ are positive, such that 

\[\big| a\big(\sigma_{1/7}(A) + \sigma_{2/7}(A) +\sigma_{3/7}(A)\big) - b\big(\sigma_{1/7}(B) + \sigma_{2/7}(B) +\sigma_{3/7}(B)\big)\big| \le 6g.\]
\end{theorem}

\begin{proof}  In the proof that appears in~\cite{MR2195064}, there is a knot $J$ such that $A = J = -B$.  The sum that appears is simpler, just a positive multiple of $ \sigma_{1/7}(J) + \sigma_{2/7}(J) +\sigma_{3/7}(J) $.  The proof however carries over almost verbatim in the current situation.
\end{proof}

\subsection{Putting the results together}
According to~\cite{MR2054808}, the values of the functions $\{\sigma_{1/7}, \sigma_{2/7}, \sigma_{1/7}, \sigma_{1/3}\}$ are linearly independent on the concordance group and can be chosen to be arbitrary even integers.  Thus, to ensure that Theorem~\ref{thm:12}  provides the desired bound on $g_4\big(nK(A,B)\big)$, we first select $A$ so that $\sigma_{1/3}(A) $ is large enough that  
\[ \big| \sum_{i=1}^{n} \sigma_\lambda\tau(K, \chi_i) +
  2a\sigma_{  1/3}(A) \big| >4g\] for all possible values of    $\sum_{i=1}^{n} \sigma_\lambda\tau(K, \chi_i) $ 
  and for all $0<a \le n$.  We then select $B$ so that $\sigma_{1/3}(B) $ is large enough to ensure that
  $ \big| \sum_{i=1}^{n} \sigma_\lambda\tau(K, \chi_i) +
  2a\sigma_{  1/3}(A)+
 2b\sigma_{  1/3}(B) \big| > 4g  $    for all possible values of    $\sum_{i=1}^{n} \sigma_\lambda\tau(K, \chi_i)$ and for all $a$ satisfying $0 \le a \le n$.  

Similarly,   to ensure that Theorem~\ref{thm:general}  provides the desired bound on $g_4\big( nK(A,B) \cs -nK(A,B)^r \big)$, we first select $A$ so that  in addition to satisfying the condition of $\sigma_{1/3}(A)$, it satisfies  $\big|\sigma_{1/7}(A) +\sigma_{2/7}(A) +\sigma_{3/7}(A)\big|  > 6n  $.  We then select $B$ so that , it satisfies \[\big|  \sigma_{1/7}(B) +\sigma_{2/7}(B) +\sigma_{3/7}(B)\big| > n(\big|\sigma_{1/7}(A) +\sigma_{2/7}(A) +\sigma_{3/7}(A)\big| + 6n \] while maintaining the condition on $\sigma_{1/3}(B)$.


 \section{Summary remarks}  In summary, for an arbitrary knot $K$, $0\le d(K, K^r) \le 2g_4(K)$.  We know that for reversible knots $d(K, K^r) = 0$ and that for all knots $K$ with $g_3(K) = 1$ we have $d(K, K^r) = 1$; more generally, we have shown here that for many knots, $d(K, K^r) \le 2g_4(K)- 1$.  In the reverse direction, we have constructed for each $g \ge 0$ a knot for which $g_4(K) = g = d(K, K^r)$.  Although we would conjecture that in some generic sense, $d(K, K^r) = 2g_4(K)$, we have been unable to solve the following problem.
 
 \smallskip
  
\begin{problem}
 Find a knot $K$ for which $d(K, K^r) > g_4(K)$.
 \end{problem}
 
 In the paper~\cite{2019arXiv190412014K}, topologically slice knots $K$ were constructed for which $d(K, K^r) =1$.  The following is open.
 
\begin{problem} Find a topologically slice knot   $K$ for which $d(K, K^r) \ge 2$.
  \end{problem}
  
  Our choice of the knots $A$ and $B$ depends on $g$.  Is this necessary?
  
  \begin{problem}
  Find a knot $K$ such that $g_4(nK) = n = d(nK,  nK^r)$ for all $n$.
  \end{problem}
  
  Finally, as is usual, our proof relied on the careful construction of particular knots that met all our needs.  Are there more natural examples?  As an example, the knot $8_{17}$ is not concordant to its reverse~\cite{MR1670424}, so $d(8_{17}, 8_{17}^r) \ge 1$ and it is known that $g_4(8_{17}) = 1$.  The challenging case is for knots $K$ with $g_4(K) \ge 2$.
  
   \begin{problem}
 Are there any low-crossing number knots $K$ for which $g_4(K) \ge 2 $ and  $d(K, K^r)\ge g_4(K)$?
  \end{problem}
  
  KnotInfo~\cite{knotinfo} lists 360 prime knots $K$ of 12 or fewer crossings which are not reversible and which satisfy $g_4(K) \ge 2$.  Presumably one can show many satisfy $d(K, K^r) \ge 1$.  I am aware of no techniques that can applied to show that any have $d(K,K^r) \ge 2$.




\bibliography{../../BibTexComplete}

\begin{bibdiv}
\begin{biblist}

\bib{MR2975163}{article}{
      author={Baader, Sebastian},
       title={Scissor equivalence for torus links},
        date={2012},
        ISSN={0024-6093},
     journal={Bull. Lond. Math. Soc.},
      volume={44},
      number={5},
       pages={1068\ndash 1078},
         url={https://doi-org.proxyiub.uits.iu.edu/10.1112/blms/bds044},
      review={\MR{2975163}},
}

\bib{MR2054808}{article}{
      author={Cha, Jae~Choon},
      author={Livingston, Charles},
       title={Knot signature functions are independent},
        date={2004},
        ISSN={0002-9939},
     journal={Proc. Amer. Math. Soc.},
      volume={132},
      number={9},
       pages={2809\ndash 2816},
         url={https://mathscinet.ams.org/mathscinet-getitem?mr=2054808},
      review={\MR{2054808}},
}

\bib{MR3848393}{article}{
      author={Cochran, Tim},
      author={Harvey, Shelly},
       title={The geometry of the knot concordance space},
        date={2018},
        ISSN={1472-2747},
     journal={Algebr. Geom. Topol.},
      volume={18},
      number={5},
       pages={2509\ndash 2540},
         url={https://doi-org.proxyiub.uits.iu.edu/10.2140/agt.2018.18.2509},
      review={\MR{3848393}},
}

\bib{MR3286894}{article}{
      author={Cochran, Tim~D.},
      author={Davis, Christopher~W.},
      author={Ray, Arunima},
       title={Injectivity of satellite operators in knot concordance},
        date={2014},
        ISSN={1753-8416},
     journal={J. Topol.},
      volume={7},
      number={4},
       pages={948\ndash 964},
         url={https://doi-org.proxyiub.uits.iu.edu/10.1112/jtopol/jtu003},
      review={\MR{3286894}},
}

\bib{MR3622312}{article}{
      author={Feller, Peter},
       title={Optimal cobordisms between torus knots},
        date={2016},
        ISSN={1019-8385},
     journal={Comm. Anal. Geom.},
      volume={24},
      number={5},
       pages={993\ndash 1025},
         url={https://doi-org.proxyiub.uits.iu.edu/10.4310/CAG.2016.v24.n5.a4},
      review={\MR{3622312}},
}

\bib{2019arXiv191001672F}{article}{
      author={{Feller}, Peter},
      author={{Park}, JungHwan},
       title={{Genus one cobordisms between torus knots}},
        date={2019-10},
     journal={arXiv e-prints},
      eprint={arxiv.org/abs/1910.01672},
}

\bib{MR0140100}{inproceedings}{
      author={Fox, R.~H.},
       title={Some problems in knot theory},
        date={1962},
   booktitle={Topology of 3-manifolds and related topics ({P}roc. {T}he {U}niv.
  of {G}eorgia {I}nstitute, 1961)},
   publisher={Prentice-Hall, Englewood Cliffs, N.J.},
       pages={168\ndash 176},
         url={https://mathscinet.ams.org/mathscinet-getitem?mr=0140100},
      review={\MR{0140100}},
}

\bib{MR656619}{article}{
      author={Gilmer, Patrick~M.},
       title={On the slice genus of knots},
        date={1982},
        ISSN={0020-9910},
     journal={Invent. Math.},
      volume={66},
      number={2},
       pages={191\ndash 197},
         url={https://doi.org/10.1007/BF01389390},
      review={\MR{656619}},
}

\bib{MR683753}{article}{
      author={Hartley, Richard},
       title={Identifying noninvertible knots},
        date={1983},
        ISSN={0040-9383},
     journal={Topology},
      volume={22},
      number={2},
       pages={137\ndash 145},
         url={https://mathscinet.ams.org/mathscinet-getitem?mr=683753},
      review={\MR{683753}},
}

\bib{2018arXiv180904186H}{article}{
      author={{Hedden}, Matthew},
      author={{Pinzon-Caicedo}, Juanita},
       title={{Satellites of Infinite Rank in the Smooth Concordance Group}},
        date={2018-09},
     journal={arXiv e-prints},
      eprint={arxiv.org/abs/1809.04186},
}

\bib{MR2195064}{article}{
      author={Kim, Se-Goo},
      author={Livingston, Charles},
       title={Knot mutation: 4-genus of knots and algebraic concordance},
        date={2005},
        ISSN={0030-8730},
     journal={Pacific J. Math.},
      volume={220},
      number={1},
       pages={87\ndash 105},
         url={https://mathscinet.ams.org/mathscinet-getitem?mr=2195064},
      review={\MR{2195064}},
}

\bib{2019arXiv190412014K}{article}{
      author={{Kim}, Taehee},
      author={{Livingston}, Charles},
       title={{Knot reversal acts non-trivially on the concordance group of
  topologically slice knots}},
        date={2019-04},
     journal={arXiv e-prints},
      eprint={arxiv.org/abs/1904.12014},
}

\bib{MR1670424}{article}{
      author={Kirk, Paul},
      author={Livingston, Charles},
       title={Twisted knot polynomials: inversion, mutation and concordance},
        date={1999},
        ISSN={0040-9383},
     journal={Topology},
      volume={38},
      number={3},
       pages={663\ndash 671},
         url={https://mathscinet.ams.org/mathscinet-getitem?mr=1670424},
      review={\MR{1670424}},
}

\bib{MR1395778}{article}{
      author={Kuperberg, Greg},
       title={Detecting knot invertibility},
        date={1996},
        ISSN={0218-2165},
     journal={J. Knot Theory Ramifications},
      volume={5},
      number={2},
       pages={173\ndash 181},
         url={https://mathscinet.ams.org/mathscinet-getitem?mr=1395778},
      review={\MR{1395778}},
}

\bib{MR547456}{inproceedings}{
      author={Litherland, R.~A.},
       title={Signatures of iterated torus knots},
        date={1979},
   booktitle={Topology of low-dimensional manifolds ({P}roc. {S}econd {S}ussex
  {C}onf., {C}helwood {G}ate, 1977)},
      series={Lecture Notes in Math.},
      volume={722},
   publisher={Springer, Berlin},
       pages={71\ndash 84},
      review={\MR{547456}},
}

\bib{knotinfo}{misc}{
      author={Livingston, Charles},
      author={Moore, Allison~H.},
       title={Knotinfo: Table of knot invariants},
         how={URL: \url{knotinfo.math.indiana.edu}},
        date={Current Year},
}

\bib{MR3784230}{article}{
      author={Miller, Allison~N.},
      author={Piccirillo, Lisa},
       title={Knot traces and concordance},
        date={2018},
        ISSN={1753-8416},
     journal={J. Topol.},
      volume={11},
      number={1},
       pages={201\ndash 220},
         url={https://doi-org.proxyiub.uits.iu.edu/10.1112/topo.12054},
      review={\MR{3784230}},
}

\bib{MR0515288}{book}{
      author={Rolfsen, Dale},
       title={Knots and links},
   publisher={Publish or Perish, Inc., Berkeley, Calif.},
        date={1976},
        note={Mathematics Lecture Series, No. 7},
      review={\MR{0515288}},
}

\bib{MR1193540}{article}{
      author={Rudolph, Lee},
       title={Quasipositivity as an obstruction to sliceness},
        date={1993},
        ISSN={0273-0979},
     journal={Bull. Amer. Math. Soc. (N.S.)},
      volume={29},
      number={1},
       pages={51\ndash 59},
  url={https://doi-org.proxyiub.uits.iu.edu/10.1090/S0273-0979-1993-00397-5},
      review={\MR{1193540}},
}

\bib{MR0158395}{article}{
      author={Trotter, H.~F.},
       title={Non-invertible knots exist},
        date={1963},
        ISSN={0040-9383},
     journal={Topology},
      volume={2},
       pages={275\ndash 280},
         url={https://mathscinet.ams.org/mathscinet-getitem?mr=0158395},
      review={\MR{0158395}},
}

\end{biblist}
\end{bibdiv}
\bibliographystyle{plain}

\end{document}